\documentclass{amsart}
\usepackage{amssymb,amsthm,tikz}
\newtheorem{proposition}{Proposition}
\newtheorem{conjecture}{Conjecture}

\def\Q{\mathbb Q}
\def\Schub{\mathfrak S}
\title{A Markov chain on the symmetric group which is Schubert positive?}
\author{Thomas Lam}
\address{Department of Mathematics, University of Michigan, Ann Arbor, MI
  48109 USA}
\email{tfylam@umich.edu}
 \thanks{TL was supported by NSF grant DMS-0901111, and by a Sloan Fellowship.}
\author{Lauren Williams}
\address{Department of Mathematics, University of California, Berkeley, CA
94705 USA}
\email{williams@math.berkeley.edu}
\thanks{LW was supported by NSF grant DMS-0854432, and by a Sloan Fellowship.}

\begin{document}
\begin{abstract}
We study a multivariate Markov chain on the symmetric group with remarkable enumerative 
properties.  We conjecture that the stationary distribution of this Markov chain can be expressed 
in terms of positive sums of Schubert polynomials.  
This Markov chain is a multivariate 
generalization of a Markov chain introduced by the first author 
in the study of random affine Weyl group elements.
\end{abstract}
\maketitle

\section{A Markov chain on the symmetric group}
Let $S_n$, $n \geq 3$ denote the symmetric group on $n$ letters and let $(i, j)$ denote the transposition which swaps $i$ and $j$.  We use conventions so that left multiplication acts on values and right multiplcation acts on positions.

Define a matrix $P = (p_{w,v})_{w,v \in S_n}$
$$
p_{w,v} = \begin{cases}
x_{w^{-1}(i+1)}& \mbox{if $v = (i ,i+1)w < w$.}\\
x_{w^{-1}(1)}& \mbox{if $v = (1,n) w > w$.} \\
* & \mbox{if $w = v$.} \\
0 & \mbox{otherwise.}
\end{cases}
$$
where $*$ is chosen so that $\sum_{v \in S_n} p_{w,v} = 1$ for each $w \in S_n$.  If the $x_i$'s are nonnegative real numbers summing to at most 1, then we can think of $P$ as defining a Markov chain on $S_n$.  When we set $x_i = 1/n$, we obtain the Markov chain defined in \cite[Section 3]{L}.  

\begin{proposition}\label{P:null}
The matrix $P^T - I$ has a one-dimensional nullspace for generic values of $x$.
In particular, when the $x_i$'s are nonnegative real numbers summing to at most $1$,
the Markov chain defined by $P$ has a unique stationary distribution.
\end{proposition}
\begin{proof}
When all $x_i$ are positive and sum to at most $1$, then it follows from  \cite[Proposition 1]{L} that we have an irreducible and aperiodic Markov chain on $S_n$, and thus we have a unique invariant distribution.  If we treat $x_1,\ldots,x_{n-2}, x_{n-1}$ as variables, then a basis of the nullspace of $P^T-I$ can be written as a rational function in the $x_i$.  This nullspace must be one-dimensional.
\end{proof}

Let $\{\zeta(w) \in \Q(x_1,x_2,\ldots,x_{n-1}) \mid w \in W\}$ denote a vector spanning the nullspace of Proposition \ref{P:null}, which we normalize by setting
$$
\zeta(w_0) =x_1^{1+2+\cdots + n-2}x_2^{1+2 + \cdots + n-3} \cdots x_{n-2}.
$$

Suppose $w = w_1 w_2 \cdots w_n \in S_n$.  Let $\chi(w) = (w_1+1) (w_2+1) \cdots (w_n+1) \in S_n$ be the cyclic shift of $w$, where the letters of $\chi(w)$ are interpreted modulo $n$.  The following follows immediately from the definitions.

\begin{proposition}\label{P:shift}
For each $w \in W$, we have $\zeta(\chi(w)) = \zeta(w)$.
\end{proposition}

\section{Schubert polynomials}
We fix notations concerning Schubert polynomials.  Let $\partial_i$ denote the divided difference operator on polynomials in $x_1,x_2,\ldots$, defined by 
$$
\partial_i f(x_1,x_2,\ldots) = \frac{f(x_1,\ldots,x_i,x_{i+1},\ldots) - f(x_1,\ldots,x_{i+1},x_{i},\ldots)}{x_i - x_{i+1}}.
$$
For the longest permutation 
$w_0 \in S_n$, we first define $$\Schub_{w_0}(x_1,x_2,\ldots) := x_1^{n-1}x_2^{n-2} \cdots x_{n-1}.$$  Next for $w \in S_n$, we let $w^{-1}w_0 = s_{i_1}s_{i_2} \cdots s_{i_\ell}$ be a reduced expression.  Then
$$
\Schub_w := \partial_{i_1} \partial_{i_2} \cdots \partial_{i_\ell} \Schub_{w_0}.
$$
The polynomial $\Schub_w$ does not depend on the choice of reduced expression.  Furthermore, $\Schub_w$ does not depend on which symmetric group $w$ is considered an element of.

\section{Conjectures}
Our main conjecture is

\begin{conjecture}\label{C:main}
In increasing strength:
\begin{enumerate}
\item
Each $\zeta(w)$ is a polynomial.
\item
Each $\zeta(w)$ is a polynomial with nonnegative integer coefficients.
\item
Each $\zeta(w)$ is a nonnegative integral sum of Schubert polynomials $\Schub_u(x_1,x_2,\ldots)$.
\end{enumerate}
\end{conjecture}

Let $\eta(w)$ denote the largest monomial that can be factored out of $\zeta(w)$.  By Proposition \ref{P:shift}, $\eta(w) = \eta(\chi(w))$.  Write $[m]$ to denote $\{0,1,2,\ldots,m\}$.
\begin{conjecture}[Monomial factor]
Assume Conjecture \ref{C:main}(1).  The map $w \mapsto \eta(w)$ is an $n$-to-$1$ map from $S_n$ to
$$
\left\{x_1^{a_1+a_2+\cdots + a_{n-2}}x_2^{a_2 + \cdots + a_{n-2}} \cdots x_{n-2}^{a_{n-2}} \mid (a_1,a_2,\ldots,a_{n-2}) \in [n-2] \times[n-3] \times \cdots \times [1]\right\}.
$$
Moreover, $\eta(w)=x_1^{a_1+a_2+\cdots + a_{n-2}}x_2^{a_2 + \cdots + a_{n-2}} \cdots x_{n-2}^{a_{n-2}}$ is given by 
$$
a_{i} = 
\#\{k \in [i+2,n] \mid w_k \in [w_i,w_{i+1}]\},$$ where $[w_i,w_{i+1}]$ denotes a cyclic subinterval of $[n]$.
\end{conjecture}

\begin{conjecture}[Special value]\
$$\zeta({\rm id}) = \Schub_{123 \cdots n} \Schub_{1n23 \cdots (n-1)} \Schub_{1(n-1) n 23\cdots (n-2)} \cdots \Schub_{134 \cdots n2}$$

\end{conjecture}

\begin{conjecture}[Special Schubert factors]
Consider the letters of $w\in S_n$ in (cyclic) order.  
If there is an adjacent string 
of letters $1,2$, then $\zeta(w)$ is a multiple of the Schubert polynomial
$\Schub_{1 3 4 5 \dots n 2}$.
More generally, if there is an adjacent string 
of letters $1,2,3,\dots,k$, then $\zeta(w)$ is a multiple of the Schubert polynomial
$\Schub_{1(k+1)(k+2)\dots n 2 3 \dots k}$.

\end{conjecture}


\section{Data}
We provide experimental data supporting these conjectures.

\subsection{$n=3$}
See Figure \ref{fig:S3}.
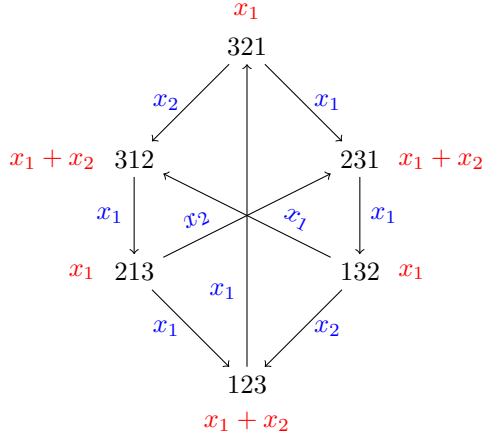
\begin{figure}[h]
\begin{center}

\begin{tikzpicture}[scale=1.5]
\node (0) [label={[red]below: $x_1+x_2$}] at (0,0) {$123$};
\node (12)  [label={[red]right: $x_1+x_2$}] at (1,2) {$231$};
\node (21)  [label={[red]left: $x_1+x_2$}] at (-1,2) {$312$};
\node (1) [label={[red]left: $x_1$}] at (-1,1) {$213$};
\node (2) [label={[red]right: $x_1$}] at (1,1) {$132$};
\node (121) [label={[red]above: $x_1$}] at (0,3) {$321$};

\draw[->] (0) -- node[near start,left=0.5pt,blue]{$x_1$} (121);
\draw[->](1)--node[near start,sloped, above=0.5pt,blue]{$x_2$}(12);
\draw[->](2)--node[near start,sloped, above=0.5pt,blue]{$x_1$} (21);
\draw[->](121)--node[left=0.5pt,blue]{$x_2$}(21);
\draw[->](121)--node[right=0.5pt,blue]{$x_1$}(12);
\draw[->](21)--node[left=0.5pt,blue]{$x_1$}(1);
\draw[->](12)--node[right=0.5pt,blue]{$x_1$}(2);
\draw[->](1) -- node[left=0.5pt,blue]{$x_1$}(0);
\draw[->](2)--node[right=0.5pt,blue]{$x_2$}(0);
\end{tikzpicture}
\caption{The transition matrix on $S_3$ (in blue) with the transitions from a vertex to itself removed and the normalized stationary distribution $\zeta$ (in red).}
\label{fig:S3}
\end{center}
\end{figure}
\subsection{$n=4$}

Using Proposition \ref{P:shift}, we need only provide data for permutations $w$ where $w_1= n$.  In the following we use $a = x_1$, $b = x_2$,  and $c = x_3$.  
We also write the answers as products of Schubert polynomials.  
Since a product of Schubert polynomials is also a nonnegative linear combination of Schubert polynomials this supports Conjecture \ref{C:main}(3).

\begin{center}
\begin{tabular}{|c|c|c|}
\hline
$w$ & $\zeta(w)$ & \\
\hline
4123&$(a^2+ab+b^2)(ab+ac+bc) $&$\Schub_{1423} \Schub_{1342}$\\
4132&$(a^2+ab+b^2)ab$&$\Schub_{1423}\Schub_{231}$\\
4213&$(a+b+c)a^2b$&$\Schub_{1243} \Schub_{321} $\\
4231&$(a^2b+a^2c+ab^2+abc+b^2c)a$&$\Schub_{1432}\Schub_{21}$\\
4312&$(ab+ac+bc)a^2$&$\Schub_{1342}\Schub_{312}$\\
4321&$a^3 b$& $\Schub_{4213}$ \\
\hline
\end{tabular}
\end{center}

Note that $a^2b+a^2c+ab^2+abc+b^2c$ is the only non-trivial factor which is not a symmetric polynomial.  

\subsection{$n=5$}
For $n=5$ we write our answers as products and sums of Schubert polynomials, multiplied
by the monomial factor $\eta(w)$.

\begin{center}
\begin{tabular}{|c|c|}
\hline
$w$ & $\zeta(w)$  \\
\hline
51234&$\Schub_{15234}\Schub_{14523}\Schub_{13452}$ \\
51243&$\Schub_{15234}\Schub_{14523}abc$\\
51324 & $\Schub_{15234}\Schub_{12453}a^2b^2c$\\
51342&$\Schub_{15234}\Schub_{14532}ab$\\
51423&$\Schub_{15234}\Schub_{13452}a^2b^2$\\
51432&$\Schub_{15234}a^3b^3c$\\
52134&$\Schub_{12534}\Schub_{13452}a^3b^2$\\
52143&$\Schub_{12534}a^4b^3c$\\
52314&$(\Schub_{15432}+\Schub_{164235})a^2bc$\\
52341&$(\Schub_{1753246}+\Schub_{265314}+\Schub_{2743156}+\Schub_{356214}+\Schub_{364215}+\Schub_{365124})a$\\
52413&$(\Schub_{164325}+\Schub_{25431})a^2b$\\
52431&$\Schub_{15243}a^3b^2c$\\
53124&$(\Schub_{146325}+\Schub_{24531})a^3b$\\
53142&$\Schub_{12543}a^4b^2c$\\
53214&$\Schub_{12354}a^5b^3c$\\
53241&$\Schub_{13542}a^4b^2$\\
53412&$\Schub_{15423}\Schub_{13452}a^2$\\
53421&$\Schub_{15423}a^3bc$\\
54123&$\Schub_{14523}\Schub_{13452}a^3$\\
54132&$\Schub_{14523}a^4bc $\\
54213&$\Schub_{12453}a^5b^2c$\\
54231&$\Schub_{14532}a^4b$\\
54312&$\Schub_{13452}a^5b^2$\\
54321&$a^6 b^3 c$\\
\hline
\end{tabular}
\end{center}


\begin{thebibliography}{xx}
\bibitem[Lam]{L} T.~Lam, The shape of a random affine Weyl group element and  random core partitions, preprint, 2011.
\end{thebibliography}
\end{document}